\documentclass{article}

\usepackage{arxiv}

\usepackage[utf8]{inputenc} % allow utf-8 input
\usepackage[T1]{fontenc}    % use 8-bit T1 fonts
\usepackage{hyperref}       % hyperlinks
\usepackage{url}            % simple URL typesetting
\usepackage{booktabs}       % professional-quality tables
\usepackage{amsfonts}       % blackboard math symbols
\usepackage{nicefrac}       % compact symbols for 1/2, etc.
\usepackage{microtype}      % microtypography
\usepackage{lipsum}		% Can be removed after putting your text content
\usepackage{graphicx}
\usepackage{amssymb,amsthm,anysize,amsmath}
\usepackage{color}
\usepackage{array,epsfig,fancyheadings,rotating,longtable}
\usepackage{caption}
\usepackage{float}
\usepackage{ulem}
\usepackage{booktabs}
\usepackage{multirow}
\theoremstyle{plain}% Theorem-like structures provided by amsthm.sty
\newtheorem{theorem}{Theorem}[section]
\newtheorem{lemma}[theorem]{Lemma}

\theoremstyle{definition}
\newtheorem{definition}[theorem]{Definition}
\newtheorem{example}[theorem]{Example}
\theoremstyle{remark}

\title{Constructing $K-$optimal designs for different Scheff\'{e} models}

\author{ {\hspace{1mm}Haosheng~Jiang}\\
	School of Economics and Statistics\\
	Guangzhou University\\
	Guangzhou, GZ 510006 \\
	\texttt{jianghaosheng1206@163.com} \\
    \And
    {\hspace{1mm}Jiali~Chen}\\
	School of Economics and Statistics\\
	Guangzhou University\\
	Guangzhou, GZ 510006 \\
	\texttt{2112064104@e.gzhu.edu.cn} \\
	\And
    {\hspace{1mm}Chongqi~Zhang}\thanks{Corresponding author.}\\
	School of Economics and Statistics\\
	Guangzhou University\\
	Guangzhou, GZ 510006 \\
	\texttt{cqzhang@gzhu.edu.cn} \\
}

\renewcommand{\shorttitle}{Constructing $K-$optimal designs for different Scheff\'{e} models}

\hypersetup{
pdftitle={A template for the arxiv style},
pdfsubject={q-bio.NC, q-bio.QM},
pdfauthor={Haosheng~Jiang,  Chongqi~Zhang},
pdfkeywords={First keyword, Second keyword, More},
}

\begin{document}
\maketitle

\begin{abstract}
 To avoid multicollinearity in regression analysis, \cite{YeZhou2013} proposed $K-$optimality criterion.
By far the most popular models for modeling the response of a mixture experiment are the Scheff\'{e} models. 
However, there have been no reports about constructing $K-$optimal designs for mixture models.
The paper constructs $K-$optimal designs for first-order and second-order Scheff\'{e} models.
The analytical solutions for first-order and second-order Scheff\'{e} models are obtained.
A series of numerical results and examples are given to illustrate the theory.

\end{abstract}

% keywords can be removed

\section{Introduction}\label{Sec1}
Mixture experiments are widely used in formulation experiments, blending experiments and marketing 
choice experiments where the goal is to determine the most preferred attribute composition of a product 
at a given price (\cite{Ragavaraoetal2011}).
The response of mixture experiment is assumed to depend only on the proportions of the ingredients in the 
mixture and not on the amount of the mixture.
In a $q$ component mixture in which $x_i$ represents the proportion of the $i$th component present 
in the mixture, these proportions are non-negative and sum to unity,
\begin{equation}\label{eq-1}
\sum_{i=1}^q x_i=1(x_i\ge 0).
\end{equation}
However constraint (\ref{eq-1}) obtains a special feature in which changes in the values of one of the components 
will lead to changes in the value of at least one of the other components.  
Then the experimental region becomes a $(q-1)-$dimensional regular simplex $S^{q-1}$ given by
\begin{center}
\begin{align*}
S^{q-1}=\left\{(x_1,x_2,\ldots,x_q)\in [0,1]^q\Big{|}\sum_{i=1}^q x_i=1,0\le x_i\le 1,i=1,\ldots,q \right\}.
\end{align*}
\end{center}

It is well known that the optimal experimental design is a useful tool for achieving maximal accuracy of 
statistical inference at minimal cost, and theoretical results concerning designs for different mixture models 
have been considered by numerous authors. \cite{GalilKiefer1977} compared with $D-$, $A-$ and $E-$optimality 
criteria of quadratic mixture model in their performance relative to these and other criteria. 
\cite{Chanetal1998} investigated $A-$optimal designs for an additive quadratic mixture model with $q\ge 3$. 
\cite{Zhangetal2005} studied $D-$ and $A-$optimal designs for an additive quadratic mixture-amount 
model. 
\cite{Zhangpeng2012} investigated the quadratic mixture canonical polynomials with spline. 
\cite{Zhangpeng2012} and \cite{Lizhang2020} constructed $D-$ and $A-$optimal designs 
for quadratic mixture canonical polynomials with spline, respectively. 
\cite{Zhangwong2013} proposed a flexible class of models for mixture experiments defined 
on $S_{*}^q=\{(z_1,\cdots,z_q)^{\prime}\in R^q|z_1+\cdots+z_q\le 1,z_i\ge 0,i=1,\cdots,q\}$, and 
investigated $D-$ and $A-$optimal designs for the models. 
\cite{Zhangetal2012} constructed two types of dual-objective optimal design for mixture experiments and 
discussed the general applicability of the design strategy to more complicated types of mixture design problems. 
Recently, \cite{Rioslin2021} proposes the order-of-addition mixture experiments, where the response depends on 
both the mixture proportions of components and their order of addition. 
\cite{Riosetal2022} studies $D-$optimal order-of-addition mixture designs by threshold accepting algorithm. 

In regression analysis, multicollinearity refers to a situation in which two or more regressions are highly 
linearly related. The effects of multicollinearity include $(i)$ the estimated regression coefficients tend to 
have large sampling variability, which implies that the estimated regression coefficients tend to vary widely 
from one sample to another; $(ii)$ the estimated individual regression coefficients may not be statistically 
significant even though a definite statistical relation exists; $(iii)$ the interpretation of regression coefficients 
is often altered. The condition number is a good measure of multicollinearity. A large condition number indicates 
severe multicollinearity (see \cite{Montgomeryetal2021}). 
To avoid multicollinearity in regression analysis, \cite{YeZhou2013} proposed $K-$optimality criterion which 
minimized the condition number of the information matrix. 
Moreover, the $K-$optimal design also minimizes the error sensitivity in the computation of the least squares 
estimator of regression parameters (\cite{Hornjohn2013}, p. 384). 
\cite{Rempelzhou2014} investigated numerical methods to compute exact $K-$optimal designs for polynomial 
regression models and linear models for factorial experiments, and implemented a simulated annealing algorithm 
to search for $K-$optimal designs on continuous design spaces. 
\cite{Yueetal2022} developed the CVX solver (\cite{Grantboyd2013}) in MATLAB to compute $K-$optimal designs for 
any regression model on a discrete design space and constructed $K-$optimal designs for polynomial, trigonometric 
and second-order response models by using the CVX. There are many results for $D-$optimal, $A-$optimal, $R-$optimal 
and other optimal designs in the literature, including theoretical and numerical results for regression models, 
design spaces and different estimators of regression coefficients. 
However, there are only a few results for $K-$optimal designs. 
This paper expands $K-$optimal design to first-order and second-order Scheff\'{e} models, and constructs $K-$optimal 
designs for the models.

The rest of the paper is organized as follows. 
Section \ref{Sec2} introduced simplex-lattice 
and simplex-centroid designs and their associated models.
The construction of $K-$optimal designs for first-order and second-order Scheff\'{e} models 
are shown in Section \ref{Sec3}. 
Section \ref{Sec3} also gives the asymptotic proporties of the total weight of vertices and midpoints of the $C(q,2)$ edges.    
A mixture experiment in formulating color photograhic dispersion is shown in Section \ref{Sec4}.
The Concluding remarks are given in Section \ref{Sec5}. 

\section{The mixture designs and their associated models}\label{Sec2}
\cite{Scheffe1958} proposed simplex-lattice design for mixture experiments and developed polynomial models.  
A polynomial model of degree $m$ in $q$ components over the simplex, the lattice, referred to as 
a $\{q,m\}-$lattice. 
Corresponding to the points in a $\{q,m\}-$lattice, the proportions used for each of 
the $q$ components have the $m+1$ equally spaced values from $0$ to $1$, that is, $x_i=0,1/m,2/m,\ldots,1$, and 
all possible mixtures with these proportions for each component are used. 
For example, the arrangement of the 
seven points of a $\{3,3\}$ simplex-lattice is presented in Figure \ref{Fig1}.
\begin{figure}[htbp]
\centering
\begin{tabular}{c}
\includegraphics[width=0.5\columnwidth]{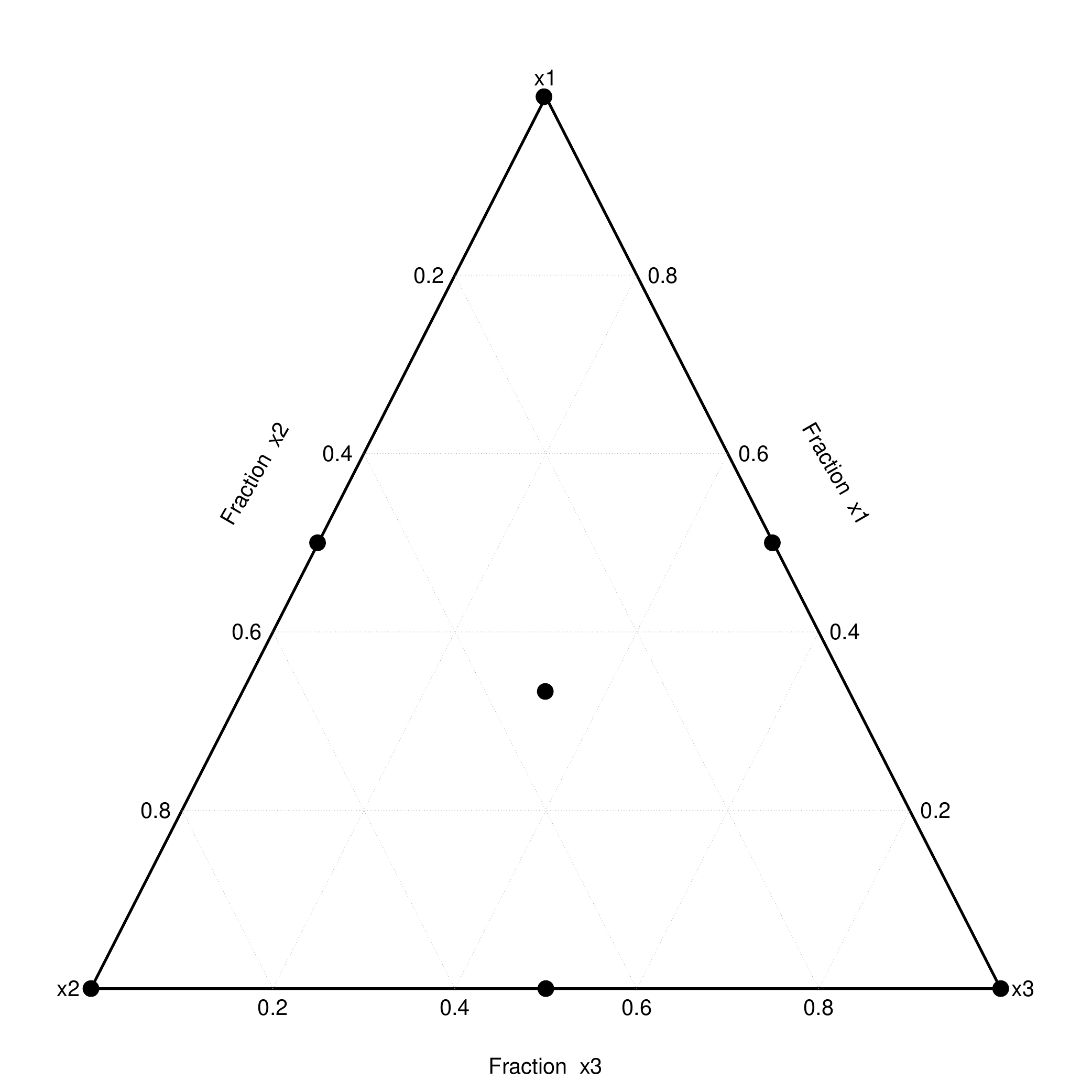}
\end{tabular}
\caption{The points of $\{3,3\}$ simplex-lattice.}\label{Fig1}
\end{figure}
The polynomial model to be fitted to data at the points of such design is 
\begin{equation}\label{eq-2}
y=\beta_0+\sum_{i=1}^q\beta_i x_i+\mathop{\sum^{q}\sum^{q}}_{i \le j}\beta_{ij}x_i x_j+\mathop{\sum^{q}\sum^{q}\sum^{q}}_{i \le j \le k}\beta_{ijk}x_i x_j x_k+\cdots+\varepsilon,
\end{equation}
where the terms up to the $m$th degree are included.
The terms in Equation (\ref{eq-2}) have meaning for us only subject to the restriction $x_1+x_2+\cdots+x_q=1$, we may make the substitution
\begin{equation}\label{eq-3}
x_q=1-\sum_{i=1}^{q-1}x_i
\end{equation}
in Equation (\ref{eq-2}), thereby removing the dependency among the $x_i$ terms, and this will not affect the 
degree of the polynomial. Since we do not wish to sacrifice information on component $q$, we do not use Equation (\ref{eq-3}).
Instead we use another approach to derive an equation in place of Equation (\ref{eq-2}) to represent the surface.
An alternative equation to Equation (\ref{eq-2}) for a polynomial of degree $m$ in $q$ components, subject to the 
restriction on the $x_i's$ in Equation (\ref{eq-1}), is derived by multiplying some of the terms in Equation (\ref{eq-2}) by identity $x_1+x_2+\cdots+x_q=1$ and simplifying.
For example, the first-degree polynomial model is given by 
\begin{equation}\label{eq-4}
y=\beta_0+\sum_{i=1}^q\beta_i x_i+\varepsilon
\end{equation}
and upon multiplying the $\beta_0$ term by $x_1+x_2+\cdots+x_q=1$, the resulting equation is 
\begin{equation}\label{eq-5}
y=\beta_0\Big{(}\sum_{i=1}^q x_i\Big{)}+\sum_{i=1}^q\beta_i x_i=\sum_{i=1}^q \beta_i^{*}x_i+\varepsilon,
\end{equation}
where $\beta_i^{*}=\beta_0+\beta_i$ for all $i=1,2,\ldots,q$.
Model (\ref{eq-5}) is also called the first-order Scheff\'{e} model.
Now that model has only the linear effect of each ingredient, such as the model (\ref{eq-4}), cannot 
capture curvature in the relationship between the ingredients and the response.
It is natural to look for models that can deal with curvature.
To achieve this end, we typically add two-factor interactions and pure quadratic effects to the model.
However, if the model includes all the linear and two-factor interaction effects for each mixture component, then 
the pure quadratic effects are redundant. To see this, note that 
\begin{equation}\label{eq-6}
x_i^2=x_i\Big{(}1-\sum_{j=1\atop j\neq i}^q x_j\Big{)}=x_i-\sum_{j=1\atop j\neq i}x_i x_j,
\end{equation}
for every ingredient proportion $x_i$. 
The second-degree polynomial model is given by 
\begin{equation*}
y=\beta_0+\sum_{i=1}^q \beta_i x_i+\mathop{\sum^{q}\sum^{q}}_{i\le j}\beta_{ij}x_i x_j+\varepsilon,
\end{equation*}
and upon multiplying the $\beta_0$ term by $x_1+x_2+\cdots+x_q=1$ and Equation (\ref{eq-6}), the resulting equation is  
\begin{align}\label{eq-7}
y&=\beta_0\Big{(}\sum_{i=1}^q x_i\Big{)}+\sum_{i=1}^q\beta_i x_i+\mathop{\sum^{q}\sum^{q}}_{i<j}\beta_{ij}x_i x_j+\sum_{i=1}^q \beta_{ii}\Big{(}x_i-\sum_{j=1\atop j\neq i}^q x_i x_j\Big{)}+\varepsilon\notag\\
    &=\sum_{i=1}^q (\beta_0+\beta_i+\beta_{ii})x_i+\sum_{i=1}^q\beta_i x_i+\mathop{\sum^{q}\sum^{q}}_{i<j}\beta_{ij}x_i x_j-\sum_{i=1}^q \beta_{ii}x_i\sum_{j=1\atop j\neq i}^q x_j+\varepsilon\notag\\
    &=\sum_{i=1}^q \beta_i^{*}x_i+\mathop{\sum^{q}\sum^{q}}_{i<j}\beta_{ij}^{*}x_i x_j+\varepsilon,
\end{align}
where $\beta_0^{*}=\beta_0+\beta_i+\beta_{ii}$, $\beta_{ij}^{*}=\beta_{ij}-2\beta_{ii}$ for all $i,j=1,2,\ldots,q$.
Model (\ref{eq-7}) is also called the second--order Scheff\'{e} model.

\cite{Scheffe1963} proposed simplex-centroid design and its associated model. 
The simplex-centroid design consists of $2^q-1$ points, that is, the $q$ permutations of $(1,0,0,\ldots,0)$, 
the $C(q,2)$ permutations of $(\frac{1}{2},\frac{1}{2},0,\ldots,0)$, the $C(q,3)$ permutations 
of $(\frac{1}{3},\frac{1}{3},\frac{1}{3},0,\ldots,0)$,$\ldots$, with finally the overall centroid 
point $(\frac{1}{q},\frac{1}{q},\ldots,\frac{1}{q})$. 
The corresponding model to be fitted to data at the points of simplex-centroid design is 
\begin{equation*}
y=\sum_{i=1}^q \beta_i x_i+\mathop{\sum^q\sum^q}_{i<j} \beta_{ij}x_ix_j+\mathop{\sum^q\sum^q\sum^q}_{i<j<k} \beta_{ijk}x_ix_jx_k+\cdots+\beta_{123\ldots q}x_1x_2x_3\ldots x_q+\varepsilon. 
\end{equation*}   

\section{$K-$optimal designs for first-order and second-order Scheff\'{e} models}\label{Sec3}
We consider the common linear regression model 
\begin{equation}\label{eq-8}
y=\mathbf{f}^{\mathrm{T}}\mathbf{(x)}\mathbf{\theta}+\varepsilon(\mathbf{x}),
\end{equation}
where $\mathbf{x}=(x_1,x_2,\ldots,x_q)^{\mathrm{T}}$ is a $q-$dimensional vector of predictors which in the design 
space $\mathcal{X}\subset S^{q-1}$, $\mathbf{f(x)}=(\mathbf{f(x_1)},\mathbf{f(x_2)},\ldots,\mathbf{f(x_p)})^{\mathrm{T}}$ 
is $p-$dimensional vector known linearly independent regression functions, $\mathbf{\theta}=(\theta_1,\theta_2,\ldots,\theta_p)^{\mathrm{T}}\in \mathbb{R}^p$ denotes the vector of unknown parameters 
and $\varepsilon(\mathbf{x})$ is assumed to be independent of $\mathbf{x}$ and independent identically distributed 
with a normal distributed with mean 0 and constant variance $\sigma^2$.
We assume that the experimenter can take $n$ independent observations of the form $y_i=\mathbf{f(x_i)^{\mathrm{T}}\theta}+\varepsilon_i(i=1,\ldots,n)$ at experimental points $\mathbf{x_1},\ldots,\mathbf{x_n}$.

An approximate design $\xi$ is a probability distribution with finite support on the factor space $\mathcal{X}$ and is represented by 
\begin{align*}
\xi(\mathbf{w})=\left(
\begin{array}{cccc}
\mathbf{x_1}& \mathbf{x_2}& $\ldots$& \mathbf{x_n}\\
\omega_1           & \omega_2           & $\ldots$& \omega_n\\
\end{array}
\right),
\end{align*}
where $\mathbf{x_i}$ denotes a support point at which a measurement is taken, and the weight 
vector $\mathbf{w}=(\omega_1,\omega_2,\ldots,\omega_n)$ with the weights satisfying $\omega_i>0$ and $\sum_{i=1}^n \omega_i=1$.
The information matrix of a design $\xi$ for model (\ref{eq-8}) as
\begin{equation*}
\mathbf{M(w)}=\sum_{i=1}^n \omega_i\mathbf{f(x_i)}\mathbf{f^{\mathrm{T}}(x_i)}.
\end{equation*}
Matrix $\mathbf{M(w)}$ is $p\times p$, always symmetric and positive semi-definite. 
One can then define the condition number of the information matrix as 
\begin{equation}
\kappa(\mathbf{M(w)})=
\begin{cases}\label{eq-9}
\frac{\lambda_{\max}(\mathbf{M(w)})}{\lambda_{\min}(\mathbf{M(w)})}, &\mbox{if}\ \mathbf{M(w)}\ \mbox{is}\ \mbox{nonsingular},\\
\infty, &\mbox{if}\ \mathbf{M(w)}\ \mbox{is}\ \mbox{singular}. \\
\end{cases}
\end{equation}
where $\lambda_{\max}(\mathbf{M(w)})$ and $\lambda_{\min}(\mathbf{M(w)})$ are the largest and the smallest eigenvalues, respectively.

The following definition, due to \cite{YeZhou2013}, provides the $K-$optimal criterion for a design belonging to $\Xi$, where $\Xi$ denote the set of all designs with nonsingular information matrix on $\mathcal{X}$.
\begin{definition}\label{def1}
A design $\xi^{*}\in \Xi$ is called $K-$optimal design for the model (\ref{eq-8}) if it minimizes 
\begin{equation}\label{eq-10}
\phi(\mathbf{w})=\frac{\lambda_{\max}(\mathbf{M(w)})}{\lambda_{\min}(\mathbf{M(w)})}
\end{equation}
over $\Xi$.
\end{definition}

By Definition \ref{def1}, the $K-$optimal design problem is given by
\begin{equation}
\begin{cases}\label{eq-11}
\arg\underset{\mathbf{w}}{\min} \phi(\mathbf{w})\\
\mbox{subject}\ \mbox{to}:\mathbf{w}\ge 0,\sum_{i=1}^n \omega_i=1,\\
\end{cases}
\end{equation}
where $\mathbf{w}\ge 0$ means that each entry of $\mathbf{w}$ is non-negative. Then $\xi(\mathbf{w}^{*})$ is called a $K-$optimal design if $\mathbf{w}^{*}$ is a solution to (\ref{eq-11}).
To construct $K-$optimal designs for the first-order Scheff\'{e} model, we firstly give the following Lemma \ref{lem1}. 
\begin{lemma}\label{lem1}
In the first-order Scheff\'{e} polynomial with regression function
\begin{equation*}
\mathbf{f(x)}=(x_1,x_2,\ldots,x_q)^{\mathrm{T}}
\end{equation*}
the vertices of $(q-1)-$dimensional regular simplex, whose vertices lie on the design 
region $S^{q-1}$, constitute the support points of an optimum design on the regular simplex $S^{q-1}$.
The corresponding design $\xi(\mathbf{w})$ has the weight vector with equal 
components $\mathbf{w}=(\frac{1}{q},\frac{1}{q},\ldots,\frac{1}{q})$ as the optimal design.
\end{lemma}
\begin{proof}
Let $\xi$ be an arbitrary design with probability measure on the regular simplex $S^{q-1}$.
The proof for the case of $q=1$ is clear.
We consider the case of $q\ge 2$.
For all $i\in \{1,2,\ldots,q\}$, let $(X_1,\ldots,X_i,\ldots,X_q) \thicksim \xi$.
$\xi$ is invariant with respect to permutations of the components.
It follows $(X_1,\ldots,X_i,\ldots,X_q),(X_i,\ldots,X_1,\ldots,X_q)\thicksim \xi$.
So we have $\omega_1=\omega_i,i=2,\ldots,q$, $i.e.$, $\omega_1=\omega_2=\cdots=\omega_q$.
By $\mathbf{w}=(\omega_1,\ldots,\omega_q)\ge 0$ and $\sum_{i=1}^q \omega_i=1$, we have $\omega_1=\omega_2=\cdots=\omega_q=\frac{1}{q}$.
This completes the proof.
\end{proof}
\begin{theorem}\label{thm1}
The $K-$optimal design $\xi$ for the first-order Scheff\'{e} model (\ref{eq-5}) has 
the weight vector $\mathbf{w}^{*}=(\frac{1}{q},\frac{1}{q},\cdots,\frac{1}{q})$ on 
the vertices of the regular simplex $S^{q-1}$.    
\end{theorem}

\begin{proof}
By Lemma \ref{lem1}, the design $\xi$ has the weight vector $\mathbf{w}^{*}=(\frac{1}{q},\frac{1}{q},\ldots,\frac{1}{q})$ on the vertices of the regular simplex $S^{q-1}$.
The corresponding information matrix is the diagonal matrix
\begin{equation*}
\mathbf{M(w^{*})}=\mbox{diag}(\frac{1}{q},\ldots,\frac{1}{q}).
\end{equation*}
The eigenvalues of information matrix $\mathbf{M(w^{*})}$ are $\lambda_1=\lambda_2=\cdots=\lambda_q=\frac{1}{q}$.  
Then $\kappa(\mathbf{M(w^{*})})=1$.
By Definition \ref{def1}, the design $\xi$ is a $K-$optimal design.
\end{proof}

\begin{example}\label{exa1}
To illustrate the result given by Theorem \ref{thm1}, we consider a first-order Scheff\'{e} polynomial 
in four variables with regression function, that is $\mathbf{f(x)}=(x_1,x_2,x_3,x_4)^{\mathrm{T}}$, 
where $\mathbf{x}=(x_1,x_2,x_3,x_4)^{\mathrm{T}}\in S^{4-1}$.
The $K-$optimal design $\xi$ for the first-order Scheff\'{e} model (\ref{eq-5}) has the weight 
vector $\mathbf{w}^{*}=(\frac{1}{4},\frac{1}{4},\frac{1}{4},\frac{1}{4})$ on the vertices of tetrahedron, that is 
\begin{equation*}
\xi(\mathbf{w}^{*})=
\left(
\begin{array}{cccc}
(1,0,0,0)&(0,1,0,0)&(0,0,1,0)&(0,0,0,1)\\
\frac{1}{4}&\frac{1}{4}&\frac{1}{4}&\frac{1}{4}\\
\end{array}
\right).
\end{equation*}
\end{example}

The following Lemma \ref{lem2} plays a key role in the construction of $K-$optimal designs for second-order Scheff\'{e} model.

\begin{lemma}\label{lem2}
In the second-order Scheff\'{e} polynomial with regression function
\begin{equation*}
\mathbf{f(x)}=(x_1,x_2,\ldots,x_q,x_1 x_2,x_1 x_3,\dots,x_{q-1}x_q)^{\mathrm{T}}
\end{equation*}
the vertices and midpoints of the $C(q,2)$ edges of $(q-1)-$dimensional regular simplex, 
whose vertices and midpoints of the $C(q,2)$ edges lie on the design region $S^{q-1}$, constitute 
the support points of an optimum design on the regular simplex $S^{q-1}$. 
The corresponding design $\xi(\mathbf{w})$ has the weight vector $\mathbf{w}=\Big{(}\omega_1,\omega_2,\ldots,\omega_q,\omega_{q+1},\ldots,\omega_{\frac{q(q+1)}{2}}\Big{)}$ as 
the optimal design, where $\omega_1=\omega_2=\cdots=\omega_q,\omega_{q+1}=\omega_{q+2}=\cdots=\omega_{\frac{q(q+1)}{2}}$.
\end{lemma}

\begin{proof}
Let $\xi$ be an arbitrary design with probability measure on the design region $S^{q-1}$.
Let $\omega_{q+1}=\omega_{q+2}=\cdots=\omega_{\frac{q(q+1)}{2}}=0$, 
by Lemma \ref{lem1}, we have $\omega_1=\omega_2=\cdots=\omega_q$.
Let $\omega_1=\omega_2=\cdots=\omega_q=0$, we consider the case of $q\ge 2$.
For all $i\in\{1,2,\ldots,q\}$, let $(X_1,\ldots,X_i,\ldots,X_q) \thicksim \xi$.
$\xi$ is invariant with respect to permutations of the components.
It follows $(X_1,\ldots,X_i,\ldots,X_q),(X_i,\ldots,X_1,\ldots,X_q)\thicksim \xi$.
So we have $\omega_{q+1}=\omega_{q+2}=\cdots=\omega_{\frac{q(q+1)}{2}}$.
This completes the proof.
\end{proof}

\begin{theorem}\label{thm2}
The $K-$optimal design $\xi$ for the second-order Scheff\'{e} model (\ref{eq-7}) has 
the weight vector $\mathbf{w}^{*}=\Big{(}\underset{q}{\underbrace{\frac{8q-7}{q(16q-15)},\ldots,\frac{8q-7}{q(16q-15)}}},\ \underset{\frac{q(q-1)}{2}}{\underbrace{\frac{16}{q(16q-15)},\ldots,\frac{16}{q(16q-15)}}}\Big{)}$ on
the vertices and midpoints of the $C(q,2)$ edges of the regular simplex $S^{q-1}$.
\end{theorem}

\begin{proof}
By Lemma \ref{lem2}, we suppose that the weight vector $\mathbf{w}$ for any design $\xi$ on 
the vertices and midpoints of the $C(q,2)$ edges of the regular simplex $S^{q-1}$ is
$\mathbf{w}=(\underset{q}{\underbrace{r_1,\ldots,r_1}},\ \underset{\frac{q(q-1)}{2}}{\underbrace{r_2,\ldots,r_2}})$.
The corresponding information matrix is 
\begin{equation*}
\mathbf{M(w)}=
\left[
\begin{array}{ccc}
(r_1+\frac{q-2}{4}r_2)\mathbf{I}_q+\frac{r_2}{4}\mathbf{J}_q&\frac{r_2}{8}\mathbf{M}_2^{\mathbf{T}} \\
\frac{r_2}{8}\mathbf{M}_2&\frac{r_2}{16}\mathbf{I}_{\frac{q(q-1)}{2}} \\
\end{array}
\right],
\end{equation*}
where $\mathbf{I}_q$ and $\mathbf{J}_q$ are $q\times q$ identity matrix and $q\times q$ matrix with 
all elements equal to 1.
$\mathbf{M}_2$ is a $C(q,2)\times q$ matrix, such that the first $i$ elements in the first row 
of $\mathbf{M}_i$ are 1 and the remaining elements in the first row are 0 and the remaining 
$C(q,2)-1$ rows of $\mathbf{M}_2$ are the different permutations of the first row according to 
lexicographical order.
For example, when $q=4$, $\mathbf{M}_2$ is a $6\times 4$ matrix and its 1st, 2nd,$\ldots$,6th rows 
are $(1,1,0,0),(1,0,1,0),(1,0,0,1),(0,1,1,0),(0,1,0,1),(0,0,1,1)$, respectively.
By solving $\Big{|}\lambda\mathbf{I}_{\frac{q(q+1)}{2}}-\mathbf{M(w)}\Big{|}=0$, the largest and 
smallest eigenvalues of information matrix $\mathbf{M(w)}$ are
\begin{align*}
\lambda_{\max}(\mathbf{M})&=\frac{\sqrt{\frac{q^2(32q-31)}{4}r_2^2-8qr_2+64}+\frac{q}{2}r_2+8}{16q}, \\
\lambda_{\min}(\mathbf{M})&=\frac{-\sqrt{\frac{q^2(32q-31)}{4}r_2^2-8qr_2+64}+\frac{q}{2}r_2+8}{16q}. 
\end{align*} 
Then 
\begin{equation*}
\kappa(\mathbf{M})=\frac{\sqrt{\frac{q^2(32q-31)}{4}r_2^2-8qr_2+64}+\frac{q}{2}r_2+8}{-\sqrt{\frac{q^2(32q-31)}{4}r_2^2-8qr_2+64}+\frac{q}{2}r_2+8}.
\end{equation*}
By solving the problem (\ref{eq-11}), i.e.,
\begin{equation}
\begin{cases}\nonumber
\arg\underset{\mathbf{w}}{\min} \kappa(\mathbf{M})=\frac{\sqrt{\frac{q^2(32q-31)}{4}r_2^2-8qr_2+64}+\frac{q}{2}r_2+8}{-\sqrt{\frac{q^2(32q-31)}{4}r_2^2-8qr_2+64}+\frac{q}{2}r_2+8},\\
\mbox{subject}\ \mbox{to}:\mathbf{w}\ge 0, qr_1+\frac{q(q-1)}{2}r_2=1.\\
\end{cases}
\end{equation}
we can get the probability measure $r_1$ and $r_2$ as
\begin{equation}
\begin{cases}\nonumber
r_1=\frac{8q-7}{q(16q-15)},\\
r_2=\frac{16}{q(16q-15)}.
\end{cases}
\end{equation}
Hence the $K-$optimal design $\xi$ for the second-order Scheff\'{e} model (\ref{eq-7}) has 
the weight vector $\mathbf{w}^{*}=\Big{(}\underset{q}{\underbrace{\frac{8q-7}{q(16q-15)},\ldots,\frac{8q-7}{q(16q-15)}}},\ \underset{\frac{q(q-1)}{2}}{\underbrace{\frac{16}{q(16q-15)},\ldots,\frac{16}{q(16q-15)}}}\Big{)}$ on
the vertices and midpoints of the $C(q,2)$ edges of the regular simplex $S^{q-1}$.
This completes the proof.
\end{proof}

Table \ref{Table1} lists the numberical results of $K-$optimal designs for second-order Scheff\'{e} model with different $q$.
From Table \ref{Table1}, we observe that $n_1r_1$ and $n_2r_2$ are strictly decreasing and increasing functions, respectively and getting arbitrarily close to 0.5 as $q$ gets large.  
This is because $n_1r_1=\frac{8q-7}{16q-15}$ and $n_2r_2=\frac{8(q-1)}{(16q-15)}$ are strictly decreasing and increasing functions, respectively.
Moreover, we also have 
\begin{equation*}
\lim_{q\to \infty} \frac{8q-7}{16q-15}=\frac{1}{2},\lim_{q\to \infty} \frac{8(q-1)}{(16q-15)}=\frac{1}{2}.   
\end{equation*}
It implies that in $K-$optimal second-order Scheff\'{e} designs, 
vertices and midpoints of the $C(q,2)$ edges on the regular simplex $S^{q-1}$ are support points,
the total weight $n_1 r_1$ and $n_2 r_2$ of vertices and midpoints of the $C(q,2)$ edges are getting arbitrarily close to 0.5 as $q$ increases.
\begin{center}
\begin{table}[htbp]
{\caption{$K-$optimal designs for second-order Scheff\'{e} model with different $q$.\label{Table1}}}
\renewcommand{\arraystretch}{1} \tabcolsep 6.5pt
{\begin{tabular}{lllllllllllll}
    \toprule
    $q$       & 3       & 4      & 5      & 6      & 7      & 8      & 9      & 10     & $\cdots$ & $\to\infty$\\
    \midrule
    $r_1$     & 0.17171 & 0.12755& 0.10153& 0.08436& 0.07216& 0.06305& 0.05598& 0.05034& $\cdots$ & 0  \\
    $n_1$     & 3       & 4      & 5      & 6      & 7      & 8      & 9      & 10     & $\cdots$ & $\infty$\\
    $n_1 r_1$ & 0.51513 & 0.5102 & 0.50765& 0.50616& 0.50512& 0.50440& 0.50382& 0.50340& $\cdots$ & 0.5\\
    \midrule   
    $r_2$     & 0.16161 & 0.08163& 0.04923& 0.03292& 0.02356& 0.01769& 0.01378 & 0.01103& $\cdots$ & 0 \\
    $n_2$     & 3       & 6      & 10     & 15     & 21     & 28     & 36      & 45     & $\cdots$ & $\infty$\\
    $n_2 r_2$ & 0.48483 & 0.48978& 0.49230& 0.49380& 0.49476& 0.49532& 0.49608 & 0.49635& $\cdots$ & 0.5\\
    \midrule    
    $n_1+n_2$ & 6       & 10     & 15     & 21     & 28     & 36     & 45      & 55     & $\cdots$ & $\infty$\\
    \bottomrule
\multicolumn{11}{l}{$n_1$ and $n_2$ denote the number of vertices and midpoints of the $C(q,2)$ edges on the regular}\\
\multicolumn{11}{l}{simplex $S^{q-1}$, respectively.}
\end{tabular}}
\end{table}
\end{center}

\begin{example}\label{exa2}
Construct $K-$optimal designs on the vertices and midpoints of three edges of an equilateral 
triangle for second-order Scheff\'{e} polynomial in three variables:
\begin{equation*}
y_{ij}=\sum_{i=1}^3 \beta_i x_i+\mathop{\sum^3\sum^3}_{i<j}\beta_{ij}x_i x_j+\varepsilon_{ij},i,j=1,2,3,
\end{equation*}
where $y_{ij}$ is the $jth$ observation from individual $i$, $\varepsilon_{ij}$ is the random errors.
Regression function $\mathbf{f(x)}=(x_1,x_2,x_3,x_1 x_2,x_1 x_3,x_2 x_3)^{\mathrm{T}}$, $\mathbf{x}=(x_1,x_2,x_3)^{\mathrm{T}}\in S^{3-1}$.
According to Theorem \ref{thm2}, the $K-$optimal design $\xi$ for the second-order Scheff\'{e} 
model (\ref{eq-5}) has the weight vector $\mathbf{w}^{*}=(\frac{17}{99},\frac{17}{99},\frac{17}{99},\frac{16}{99},\frac{16}{99},\frac{16}{99})$ on 
the vertices and midpoints of three edges of an equilateral triangle, that is 
\begin{equation*}
\xi(\mathbf{w}^{*})=
\left(
\begin{array}{cccccc}
(1,0,0)&(0,1,0)&(0,0,1)&(\frac{1}{2},\frac{1}{2},0)&(\frac{1}{2},0,\frac{1}{2})&(0,\frac{1}{2},\frac{1}{2})\\
\frac{17}{99}&\frac{17}{99}&\frac{17}{99}&\frac{16}{99}&\frac{16}{99}&\frac{16}{99}\\
\end{array}
\right).
\end{equation*}
\end{example}

\section{A mixture experiment in formulating color photograhic dispersion}\label{Sec4}
In formulation experiments and blending experiments, each component in the mixture can be varied within 
lower and upper constraints. When each component in the mixture is restricted by lower 
bounds $0\le L_i\le x_i,i=1,2,\ldots,q$ or upper bounds $0\le x_i\le U_i,i=1,2,\ldots,q$. 
Transforming the restricted components to pseudo-components 
can be used, that is, using
\begin{equation*}
x_i^{\prime}=\frac{x_i-L_i}{1-\sum_{i=1}^q L_i},i=1,2,\ldots,q
\end{equation*}
or 
\begin{equation*}
x_i^{*}=\frac{U_i-x_i}{\sum_{i=1}^q U_i-1},i=1,2,\ldots,q.
\end{equation*}
The quantity $1-\sum_{i=1}^q L_i$ and $\sum_{i=1}^q U_i-1$ must be greater than zero so that the simplex in the $x_i^{\prime}$ and $x_i^{*}$ lies 
entirely inside the simplex in the $x_i$.
The pseudo-component space by transforming the restricted components to pseudo-components makes it easier to see details of the response surface in the restricted region.
For example, the package \textbf{mixexp} of $\mathbf{R}$ creates two contour plots of a quadratic model fit to the data of \cite{Cornell2002}'s Table 4.1.
The contour plot is made over the actual component space and pseudo-component space as shown in Figure \ref{Fig2}.
\begin{figure}[htbp]
    \begin{minipage}{0.5\linewidth}
    \centering
    \includegraphics[width=1\linewidth]{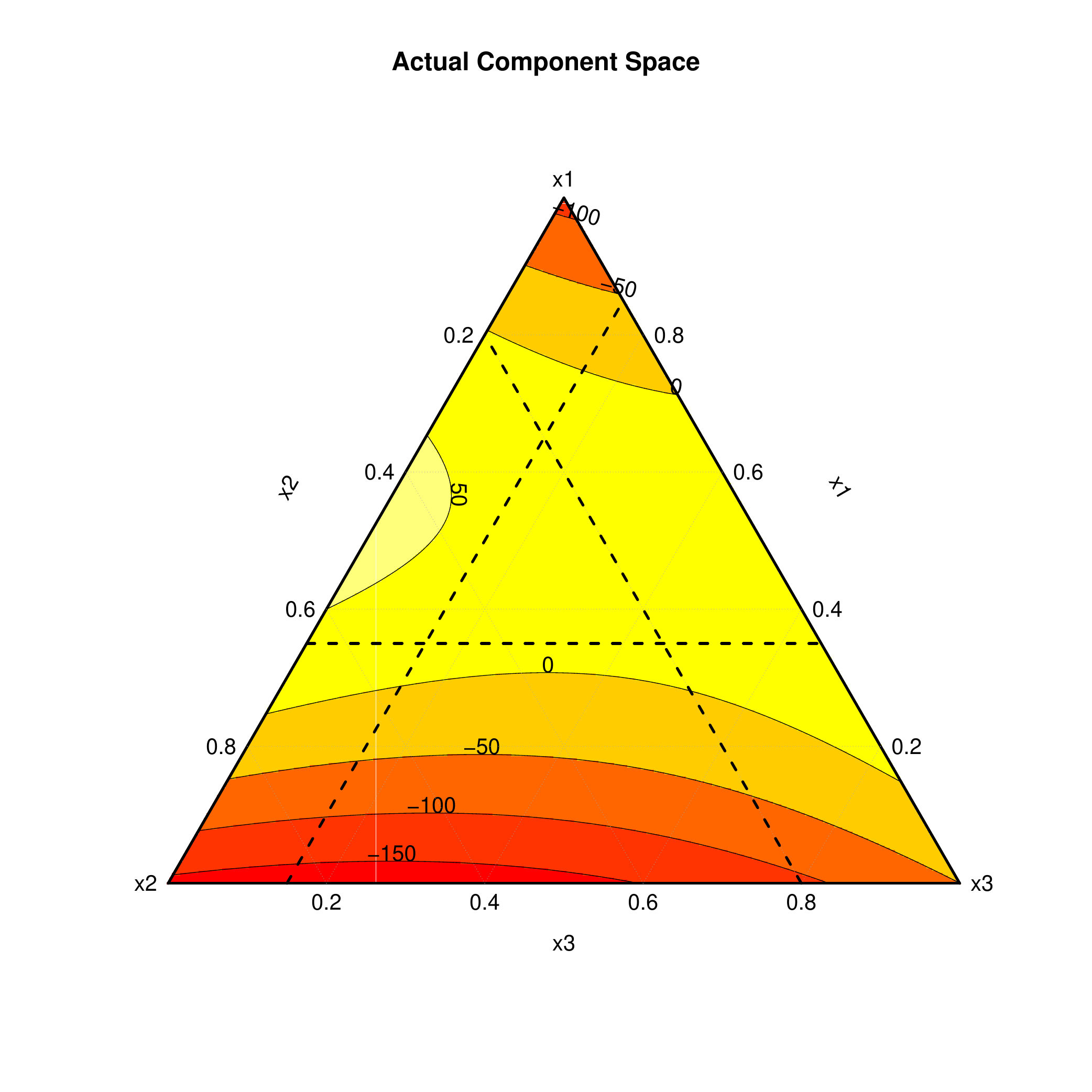}
    \end{minipage}
    \begin{minipage}{0.5\linewidth}
    \centering
    \includegraphics[width=1\linewidth]{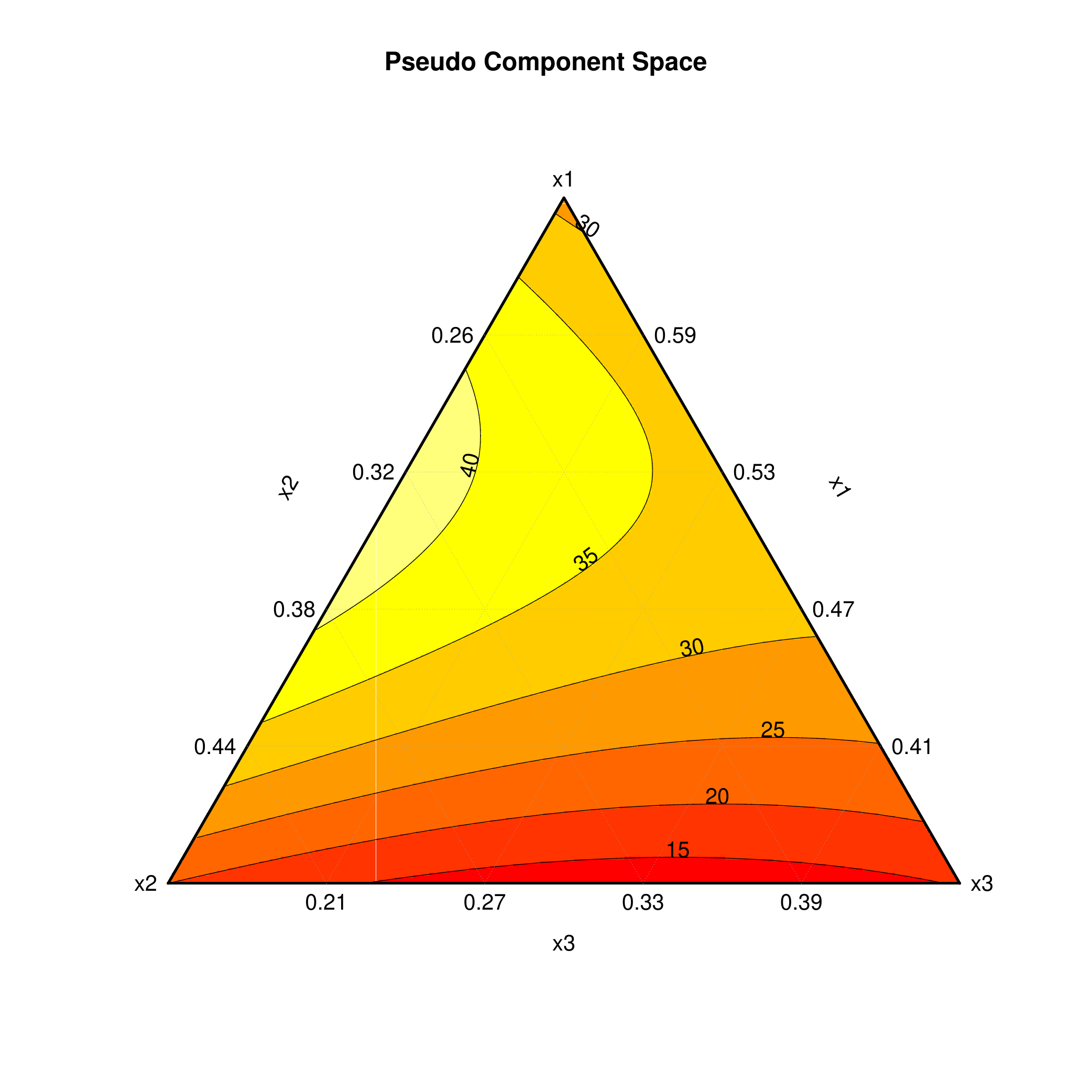}
    \end{minipage}
    \caption{Comparison of contour plots on unrestricted and pseudo-component space.}\label{Fig2}
\end{figure}

As can be seen, there is much more detail in the restricted space on the right.
The following example shows how to construct $K-$optimal designs for second-order Scheff\'{e} model 
with each component is restricted by lower and upper constraints. 

\begin{example}\label{exa3}
\cite{Smith2005} discussed a mixture experiment in formulating color photograhic dispersion that included 
constraints with upper and lower bounds on the components. The constraint on the design region are given 
below. 
\begin{align*}
0.08\le x_1\le 0.43\\
0\le x_2\le 0.35,\\
0.15\le x_3\le 0.5,
\end{align*}
where $x_1$=coupler, $x_2$=solvent, $x_3$=stabelizer.
The design region and support points are shown in Figure \ref{Fig3}.
\begin{figure}[H]
    \centering
    \begin{tabular}{c}
    \includegraphics[width=0.5\columnwidth]{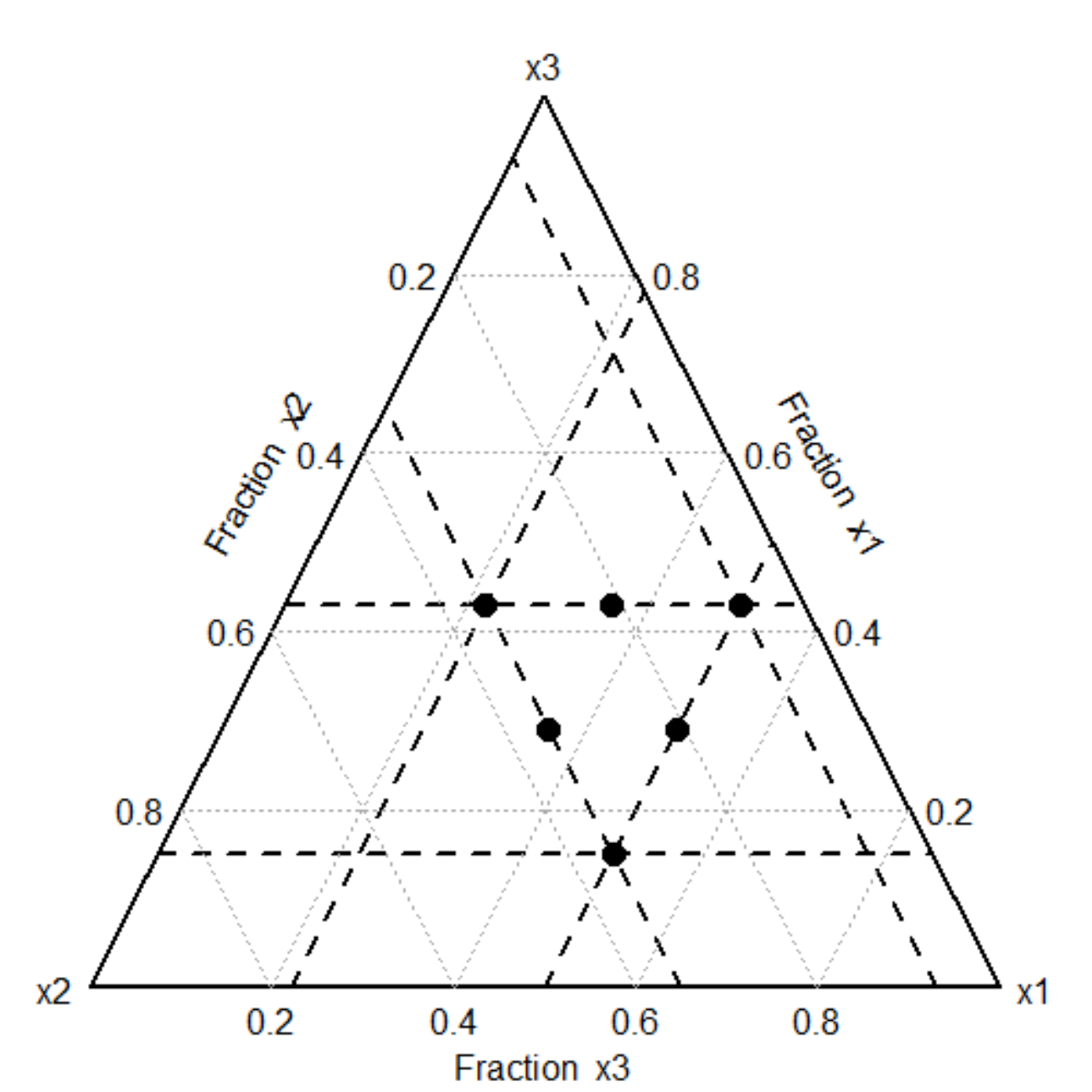}
    \end{tabular}
    \caption{The design region and support points in formulating color photograhic dispersion experiment.}\label{Fig3}
    \end{figure}
\noindent The three components in such experiment are restricted by specific lower and upper constraints.
To construct $K-$optimal design for second-order Scheff\'{e} model in such experiment, let us consider the transformation which 
transforms the restricted components to pseudo-components. The settings in the original $x_i$ components, 
corresponding to the lattice settings in the $x_i^{\prime}$, are presented in Table \ref{Table2}.
\begin{center}
\begin{table}[htbp]
{\caption{Original component settings and pseudo-component settings.\label{Table2}}}
\renewcommand{\arraystretch}{1} \tabcolsep 21pt
{\begin{tabular}{*{6}{c}}
    \toprule
    \multicolumn{3}{c}{original component settings} & \multicolumn{3}{c}{pseudo-component settings} \\
    \cmidrule(rrr){1-3}\cmidrule(rrr){4-6}
    $x_1$ & $x_2$ & $x_3$ & $x_1^{\prime}$& $x_2^{\prime}$& $x_3^{\prime}$\\
    \midrule
    0.4300001 & 0.3500002 & 0.2200000 & 1             & 0             & 0  \\
    0.4300000 & 0.0699999 & 0.5000000 & 0             & 1             & 0  \\
    0.1500001 & 0.3499999 & 0.5000000 & 0             & 0             & 1 \\
    0.4300001 & 0.2100001 & 0.3600000 & 1/2           & 1/2           & 0 \\
    0.2900001 & 0.3500002 & 0.3600000 & 1/2           & 0             & 1/2\\
    0.2900001 & 0.2099999 & 0.5000000 & 0             & 1/2           & 1/2\\
    \bottomrule
  \end{tabular}}
\end{table}
\end{center}

We now construct $K-$optimal design for second-order Scheff\'{e} model in pseudo-component space.
The $K-$optimal design problem is given by
\begin{equation*}
\begin{cases}\nonumber
\arg\underset{\mathbf{w}}{\min} \kappa(\mathbf{M})=\frac{\sqrt{\frac{2313}{4}r_2^2-24r_2+64}+\frac{3}{2}r_2+8}{-\sqrt{\frac{2313}{4}r_2^2-24r_2+64}+\frac{3}{2}r_2+8},\\
\mbox{subject}\ \mbox{to}:r_1,r_2\ge 0, 3(r_1+r_2)=1.\\
\end{cases}
\end{equation*}

By solving the above problem, we can get the probability measure $r_1$ and $r_2$ as
\begin{equation}
\begin{cases}\nonumber
r_1=\frac{17}{99},\\
r_2=\frac{16}{99}.
\end{cases}
\end{equation}
The $K-$optimal design $\xi^{\prime}(\mathbf{w}^{*})$ for second-order Scheff\'{e} model in pseudo-component space is given by
\begin{equation*}
\xi^{\prime}(\mathbf{w}^{*})=
\left\{
\begin{array}{cccccc}
(1,0,0)&(0,1,0)&(0,0,1)&(\frac{1}{2},\frac{1}{2},0)&(\frac{1}{2},0,\frac{1}{2})&(0,\frac{1}{2},\frac{1}{2})\\
\frac{17}{99}&\frac{17}{99}&\frac{17}{99}&\frac{16}{99}&\frac{16}{99}&\frac{16}{99}\\
\end{array}
\right\}.
\end{equation*}
Thus, the $K-$optimal design $\xi^{*}$ for second-order Scheff\'{e} design in original component space is given by
\begin{equation*}
\xi^{*}(\mathbf{w}^{*})=
\left\{
\begin{array}{cccccc}
(0.43,0.35,0.22)&(0.43,0.07,0.50)&(0.15,0.35,0.50)\\
\frac{17}{99}&\frac{17}{99}&\frac{17}{99}\\
(0.43,0.21,0.36)&(0.29,0.35,0.36)&(0.29,0.21,0.50)\\
\frac{16}{99}&\frac{16}{99}&\frac{16}{99}\\
\end{array}
\right\}.
\end{equation*}
\end{example} 

\section{Concluding remarks}\label{Sec5}
\cite{Yueetal2022} points that it can be hard to find $K-$optimal designs analytically 
for second-order response models. 
The paper investigated $K-$optimal designs for first-order and second-order Scheff\'{e} models.
The analytical solutions for first-order and second-order Scheff\'{e} models are obtained. 
An illustrative example is shown for constructing $K-$optimal design for second-order Scheff\'{e} model 
with each component is restricted by lower and upper constraints.

To compare the efficiency of designs, let $\xi_{D}^{*}$, $\xi_{K}^{*}$ denote the $D-$optimal 
and $K-$optimal designs for model (\ref{eq-7}), the $D-$efficiency of the design $\xi_{K}^{*}$ with 
information matrix $\mathbf{M}(\xi_{K}^{*})$ relative to the design $\xi_{D}^{*}$ with 
information matrix $\mathbf{M}(\xi_{D}^{*})$ is defined as  
\begin{equation*}
Eff_D(\xi_K^{*})=\Big{(}\frac{det\{\mathbf{M}(\xi_K^{*})\}}{det\{\mathbf{M}(\xi_D^{*})\}}\Big{)}^{\frac{1}{p}},
\end{equation*}
where $p$ is the number of unknown parameters in the model (\ref{eq-7}). The $K-$efficiency of the $D-$optimal 
design $\xi_D^{*}$ with $\kappa(\xi_D^{*})$ relative to the design $\xi_K^{*}$ with $\kappa(\xi_K^{*})$ is defined by
\begin{equation*}
Eff_K(\xi_D^{*})=\Big{(}\frac{\kappa\{\mathbf{M}(\xi_K^{*})\}}{\kappa\{\mathbf{M}(\xi_D^{*})\}}\Big{)}^{\frac{1}{p}},
\end{equation*}
where $\kappa(\mathbf{M})$ denotes the condition number of matrix $\mathbf{M}$.
We find that $Eff_D(\xi_K^{*})=99.95\%$ for $q=3$, and $Eff_K(\xi_D^{*})=99.98\%$ for $q=3$.
Our designs are thus shown to be more efficient.

An interesting research problem is to investigate $K-$optimal designs for other 
mixture models, such as Becker model, Draper-John's model and so on. In addition, algorithm 
to search $K-$optimal designs 
for mixture models is also interesting and challenging research problem.
We look forward to exploring these two problems in future research.

\bibliographystyle{unsrtnat}
%\bibliography{references}  %%% Uncomment this line and comment out the ``thebibliography'' section below to use the external .bib file (using bibtex) .

%%% Uncomment this section and comment out the \bibliography{references} line above to use inline references.

\end{document}